\documentclass[a4paper,11pt]{article}
\usepackage{amssymb,amsmath,amsthm}
\usepackage{graphics,graphicx,color}
\usepackage{srcltx}
\DeclareGraphicsExtensions{.eps} \textwidth=169mm
\definecolor{darkred}{rgb}{0.9,0.1,0.1}

\newtheorem{proposition}{Proposition}
\newtheorem{theorem}{Theorem}
\newtheorem{lemma}{Lemma}

\newtheorem{remark}{Remark}
{\rm}
\definecolor{darkred}{rgb}{0.9,0.1,0.1}

\def\eps{\varepsilon}

\definecolor{darkred}{rgb}{0.9,0.1,0.1}

\author{A. Piatnitski$^{\small\bf a,b}$ and E. Zhizhina$^{\small\bf b,a}$\\
\\
{\small $^{\small\bf a}$ UiT, The Arctic University of Norway, Campus Narvik,}\\[-1.5mm]
{\small P.O.Box 385, Narvik 8505, Norway}\\[2.5mm]
{\small $^{\small\bf b}$Institute for Information Transmission Problems of RAS,}\\[-1.5mm]
{\small 19, Bolshoi Karetnyi per. build.1,}\\[-1.5mm]
{\small 127051 Moscow, Russia}\\
$ $}

\begin{document}

\title{High-contrast periodic random jumps in continuum and limit Markov process
}

\maketitle



\begin{abstract}
The paper deals with the asymptotic properties of a random jump process in a high contrast periodic medium in $\mathbb R^d$, $d\geq 1$. 
We show that if the coordinates of the random jump process in $\mathbb R^d$ are equipped with an extra variable that characterizes the position of the process inside the period,
then the limit dynamics of this two-component process
is Markov. We describe the limit process and prove the convergence in law in the path space.
We show that the components of the limit process are coupled and
derive the evolution equation with a memory term  for the spatial component of the process.
We also discuss the construction of the limit process in the $L^2$ space and study the spectrum of the generator of the limit semigroup.

\end{abstract}

\noindent

\section*{Introduction}

We study in this work the scaling limit of random jump processes in $\mathbb R^d$, $d\geq 1$, in periodic  high-contrast media. More precisely, we assume that the transition intensities of the random jumps depend on a small parameter $\varepsilon>0$ in such a way that they are
 of order one for the points located in a given periodic open set
 in $\mathbb R^d$, and are of order $\eps^2$ in the complement of this set.   We consider a continuous time Markov process $ \mathfrak{X}_\eps(t)$ with these transition rates and say that the process is "fast" \ in the mentioned periodic
 set and "slow" in its complement.
The goal of this work is to study the scaling limit of the process $ \mathfrak{X}_\eps(t)$ and to describe the limit process.
To this end we introduce the process $X_\varepsilon(t)=\varepsilon  \mathfrak{X}_\eps(t/\varepsilon^2)$
and investigate its asymptotic behaviour  as $\varepsilon\to 0$.


The main strategy of our approach is to rearrange the process $X_\varepsilon(t)$ in such a way that the limit dynamics remains Markovian. This idea was realized in our previous works for lattice random walks \cite{JSP} and for diffusion processes \cite{AA} in high contrast environments.
It turns out that in order to keep the Markovity of the limit process
one can equip the coordinate process $X_\eps(t)$ with an additional variable, $k_\eps( X_\eps(t))$, that specifies
the position of the random walk in the period. This variable is defined as the fractional part of  $\eps^{-1}X_\eps(t)$. We denote the slow part of the period by $\overline{G}$, identify all the points of the fast
set and call them  $\star$. Then $k_\eps( X_\eps(t))$ takes on values in $G^\star = \overline{G}\cup \star$. Notice that  $G^\star$ is a compact state space.   Although in the original process $(X_\eps(t), k_\eps( X_\eps(t)))$
the last component is a function of $X_\eps(t)$, in the limit process the last component is independent
of the other component.
The limit process  is a two-component continuous time Markov process $\mathcal{X}(t) = ({X}(t), k(t))$ taking on values in a space $E=\mathbb R^d\times G^\star$. The  first component ${X}(t)$ of $\mathcal{X}(t)$ lives in the space $\mathbb R^d$, while the second component is a jump Markov process $k(t)$ with a compact  state space $G^\star$.
 The process $k(t)$
does not depend on $X(t)$.  When $k(t) = \star$,  the first
component $X(t)$ evolves along the trajectories of a Brownian motion in $\mathbb R^d$, while when $k(t) \neq \star$, then the first
component remains still until the second component of the process
takes again the value equal to $\star$. Thus the trajectories of  ${X}(t)$  coincide with  the trajectories of a Brownian diffusion in $\mathbb R^d$ during the time when  $k(t)=\star$.  As long as $k(t) \neq \star$, then ${X}(t)$ does
not move, only the second component of the process evolves.

One of the main goals of this paper is to construct the limit Markov process with sample paths in $D_E[0, \infty)$.
The Feller semigroups are usually used to construct the associated Markov processes. In Sect. 2, we define the Feller semigroup on the Banach space $C_0(E)$ of continuous functions on $E$ and study the semigroup convergence in  $C_0(E)$. It is important to note that despite the fact that the limit semigroup acts on $C_0(E)$, the approximating semigroups do not need to be defined on $C_0(E)$, see arguments for the construction of the limit Markov process in Theorem \ref{T1}.  We also study the convergence of the semigroups in $L^2(E)$ and provide the spectral analysis of the generator of the limit semigroup.

%


Our approach relies on a combination of approximation results from \cite{EK} and corrector techniques from the homogenization theory.  A crucial step here is constructing several periodic correctors which are introduced as solutions of auxiliary equations on the period.


Homogenization problems for uniformly elliptic symmetric convolution-type operators in periodic environments have been  studied in \cite{SIMA17}, similar problem in the non-symmetric case has been addressed in
\cite{PZ19}.  $\Gamma$-convergence results for nonlocal uniformly coercive convolution-type functionals were obtained in \cite{BP21}, \cite{BP22}.

For the second order elliptic and parabolic differential operators the double porosity model was originally studied in \cite{ADH90} and then, by means of the two-scale convergence method, in \cite{Ale92}. At present there is a vast literature on this topic.

The discrete double porosity model has been considered in \cite{JSP},  and by the variational approach in \cite{BCP15} .

In \cite{Zhi00} double porosity models defined on singular periodic structures were investigated; also, in the cited work
the limit behaviour of the spectrum of a high contrast elliptic operator with a periodic microstructure has been described.




The paper is organized as follows. In Section \ref{Sec1} we formulate the problem and provide all the conditions
on the coefficients of the studied operators and on the microscopic geometry. The goal of Section \ref{sec2} is to show
that the semigroups corresponding to the extended processes converge and to characterize the limit semigroup.
Here we deal with operators  whose resolvent acts in the space of functions that are continuous both on slow and fast
phases. In Section \ref{sec3}  we prove the convergence in law of the processes in the Skorokhod topology of $D_E [0, \infty)$.  The limit behaviour of the coordinate process 
is studied in Section \ref{sec4}. We emphasize that the limit coordinate process need not be Markov and derive the evolution equation with the memory term for the spatial component ${X}(t)$ of the limit process $\mathcal{X}(t)$.
In Section \ref{sec5} we consider the problem in the $L^2$ framework. Here we discuss the semigroup convergence, describe
the generator of the limit semigroup and study the spectrum of the limit semigroup generator.

\section{Problem setup}
\label{Sec1}

Consider a convolution-type  operator $A$ defined by
\begin{equation}\label{L}
A u(x) = \int_{\mathbb{R}^d} a(x-y) \Lambda(x,y) \big( u(y) - u(x) \big) dy.
\end{equation}
We assume that the convolution kernel $a(z)$ has the following properties:
\begin{equation}\label{a-1}
a(z) \ge 0, \quad  a(-z) = a(z), \quad a \in L_1(\mathbb{R}^d) \cap L^\infty (\mathbb{R}^d),
\quad \int_{\mathbb{R}^d} |z|^3 a(z) dz < \infty;
\end{equation}
\begin{equation}\label{a-1_low}
\exists \; r_0>0 \; \mbox{ and } \; c>0 \; \mbox{ such that } \;  a(z)\geq c\quad\hbox{if }|z|\leq r_0.
 \end{equation}
The finiteness of the third moment of $a$ 
specified in \eqref{a-1} is an excessively strong condition. We impose this condition to simplify the proof of the main result, namely Proposition \ref{P1}. In fact, it is sufficient to require the finiteness of the second moment $ \int_{\mathbb{R}^d} |z|^2 a(z) dz$. Then the result remains valid, however the proof is getting more complicated, see e.g. \cite{SIMA17}.

Concerning
$\Lambda(x,y)$ we suppose that it is a periodic positive function describing high-contrast environments, see  \eqref{Lambda-eps} below.

Notice that $A$ is the generator of a continuous time Markov jump process $ \mathfrak{X}(t)$ in $\mathbb{R}^d$ with the intensity of  jumps $x \to y$ equal to $ p(x,y)=a(x-y) \Lambda(x,y)$.

\medskip
In this paper we consider a family of Markov jump processes $ \mathfrak{X}_\eps(t)$ with transition intensities $p^{(\varepsilon)}(x,y)$ that depend on a small parameter $\varepsilon>0$ and, for each value of the parameter, satisfy the properties formulated above.
%
In order to introduce $p^{(\varepsilon)}(x,y)$ assume that on the periodicity cell $\mathbb T^d$ there are two open
sets $G$ and $Y$ with  Lipschitz continuous boundaries such that
\begin{equation}\label{defGY}
\mathbb T^d = \overline{G} \cup Y; \quad G,\, Y \neq \emptyset, \; G \cap Y =   \emptyset,
\end{equation}
and that, furthermore,  the periodic extension of $Y$ denoted by $Y^{\sharp}$
is unbounded and connected.
Here the connectedness is understood in terms of the  function $a(z)$; we say that two points $x',\,x''\in Y^{\sharp}$ are connected if there exists a path $x^1,\ldots,x^M$ in
$Y^{\sharp}$ such that $x^1=x'$, $x^M=x''$ and $a(x^j-x^{j+1})>0$ for all $j=1,\ldots, M-1$.
We also denote by $G^\sharp$ the periodic extension of $G$. Then
$\mathbb R^d = \overline{G}^{\sharp} \cup Y^{\sharp}$.

The sets  $\varepsilon G^{\sharp}$ and $\varepsilon Y^{\sharp}$ are obtained by the dilatation of  $G^{\sharp}$ and $Y^{\sharp}$:
$$
\varepsilon G^{\sharp}=\{x\in\mathbb R^d\,:\,\varepsilon^{-1}x\in G^{\sharp}\},\qquad
\varepsilon Y^{\sharp}=\{x\in\mathbb R^d\,:\,\varepsilon^{-1}x\in Y^{\sharp}\}.
$$
Then
\begin{equation}\label{decZ}
\mathbb{R}^d = \varepsilon\overline{G}^{\sharp} \cup \varepsilon Y^{\sharp}.
\end{equation}
Let  $L^\infty_0 (\mathbb R^d)$ be the Banach space of measurable bounded
functions on $\mathbb R^d$ vanishing at infinity: $\mathop{\mathrm{esssup}}\limits_{|x|>R}|u|\to0$ as $R\to\infty$.
This space is equipped with the norm $\|u \| = \mathop{\mathrm{esssup}}\limits_{x \in \mathbb R^d} |u(x)|$. Define the rescaled operator $A_\varepsilon$ as follows
\begin{equation}\label{Leps}
A_\varepsilon u(x) = \frac{1}{\varepsilon^{d+2}} \int_{\mathbb{R}^d} a_\varepsilon(x-y) \Lambda_\varepsilon(x,y) \big( u(y) - u(x) \big) dy, \quad u \in L^\infty_0 (\mathbb R^d),
\end{equation}
where we denote
$$
a_\varepsilon(x-y) =  a(\frac{x-y}{\varepsilon}), \quad \Lambda_{\varepsilon}(x,y) = \Lambda(\frac{x}{\varepsilon}, \frac{y}{\varepsilon}), \quad x,y \in \mathbb{R}^d = \varepsilon\overline{G}^{\sharp} \cup \varepsilon Y^{\sharp}.
$$
In what follows we identify periodic functions  with the functions defined on $d$-dimensional torus $\mathbb T^d$.

The kernel $\Lambda_\varepsilon$ describes  the high-contrast periodic structure.
Namely, assume that $ \Lambda_\varepsilon$ admits the representation
\begin{equation}\label{Lambda-eps}
\Lambda_\varepsilon (x,y) = \Lambda^0\Big(\frac x\varepsilon,\frac y\varepsilon\Big) + \varepsilon^2 v\Big(\frac x\varepsilon,\frac y\varepsilon\Big),
\end{equation}
where both $\Lambda^0(\xi,\eta)$ and $ v(\xi,\eta)$, $\ \xi = \frac x\varepsilon, \, \eta = \frac y\varepsilon$,
are periodic functions with period $[0,1)^d$:
$$
\Lambda^0(\xi,\eta) = \Lambda^0(\xi+k,\eta+n), \quad  v(\xi,\eta) =  v(\xi+k,\eta+n), \qquad \forall \; k, \, n \in \mathbb{Z}^d,
$$
and
\begin{equation}\label{Lambda0}
\Lambda^0 (\xi,\eta) = {\bf{1}}_{ Y^{\sharp}}(\xi) \cdot {\bf{1}}_{ Y^{\sharp}}(\eta),
\end{equation}

\begin{equation}\label{v-eps}
 v(\xi, \eta) = w (\xi, \eta)  \big( 1- {\bf{1}}_{Y^{\sharp}}(\xi) \cdot {\bf{1}}_{Y^{\sharp}}(\eta) \big), \;\; \xi, \eta \in \mathbb{T}^d,
\end{equation}
with
\begin{equation}\label{vv}
0< \alpha_1 \le w(\xi, \eta) \le \alpha_2 < \infty \quad \mbox{ for all } \; \xi, \eta \in \mathbb{T}^d.
\end{equation}
Here  ${\bf 1}$ stands for the characteristic function.
We also assume that
\begin{equation}\label{w}
w\in C \big(\mathbb{T}^d \times \mathbb{T}^d \big).  
\end{equation}
Denote $\Lambda^0_\varepsilon (x,y) =
\Lambda^0(\frac{x}{\varepsilon}, \frac{y}{\varepsilon})$ and $v_\varepsilon (x,y) = v(\frac{x}{\varepsilon}, \frac{y}{\varepsilon})$.
Thus representations \eqref{Lambda0} - \eqref{v-eps} imply that
$\Lambda^0_\varepsilon (x,y) = 0$, if $x\in  \varepsilon G^{\sharp}$ or  $y \in  \varepsilon G^{\sharp}$, and
$v_\varepsilon(x,y) =0$ if both $x\in  \varepsilon Y^{\sharp}$ and  $y \in  \varepsilon Y^{\sharp}$.
Owing to \eqref{Lambda-eps} the operator in \eqref{Leps} can be written as
\begin{equation}\label{Leps-sum}
A_\varepsilon = A^0_\varepsilon + V_\varepsilon,
\end{equation}
where
\begin{equation}\label{L0eps}
A^0_\varepsilon u(x) = \frac{1}{\varepsilon^{d+2}} \int_{\mathbb{R}^d} a_\varepsilon(x-y) \, \Lambda^0_\varepsilon(x,y) \, \big( u(y) - u(x) \big) dy,
\end{equation}
\begin{equation}\label{Veps}
V_\varepsilon u(x) = \frac{1}{\varepsilon^{d}} \int_{\mathbb{R}^d} a_\varepsilon(x-y) \, v_\varepsilon(x,y)\, \big( u(y) - u(x) \big) dy.
\end{equation}
For each $\varepsilon>0$, \eqref{Leps} defines a bounded linear operator  $A_\varepsilon$ on $L^\infty_0 (\mathbb{R}^d)$, and  $A_\varepsilon$ is the generator of a continuous time Markov jump process  $X_\varepsilon(t) = \varepsilon \mathfrak{X}_\eps(t/\varepsilon^2)$
on $\mathbb R^d = \varepsilon \overline{G}^{\sharp} \cup \varepsilon Y^{\sharp}$ with the intensity of jumps $ \eps^{-d-2} a_\varepsilon(x-y) \, \Lambda_\varepsilon(x,y) $. The goal of this work is to study the behaviour of this process  as $\varepsilon\to 0$, and to construct the limit semigroup and the limit process.

\section{Semigroup convergence in $C_0(E)$}
\label{sec2}

In this section we equip  the random jump process $X_\varepsilon (t)$ with an additional component,
and for the extended process we prove the convergence of the corresponding semigroups.

 For each $\varepsilon>0$ we introduce a mapping $k_\varepsilon(x)$ from $\mathbb R^d$ to $G^\star = \overline{G} \cup \{\star\}$
defined by
\begin{equation}\label{kx}
k_\varepsilon(x) \ = \ \left\{
\begin{array}{l}
\star \; \; \mbox{ if } \;  x \in \varepsilon Y^{\sharp}, \\
\big\{\frac{x}{\varepsilon} \big\} \in \overline{G} \;  \mbox{ if } \;  x \in \varepsilon \overline{G}^{\sharp}.
\end{array}
\right.
\end{equation}
With this construction in hands we  introduce the metric space
\begin{equation}\label{Ee}
E_\varepsilon \ = \  \left\{ (x, k_\varepsilon(x)), \; x \in \mathbb R^d,\; k_\varepsilon
(x) \in  G^\star \right\}, \quad E_\varepsilon \subset \mathbb R^d \times G^\star,
\end{equation}
with a metric that coincides with the metric in  $\mathbb R^d$ for the first component of $(x,k_\varepsilon(x)) \in E_\varepsilon$.

We denote by $C_0(E_\varepsilon)$ the following Banach space of functions on $E_\varepsilon$. We say that $f(x, k_\varepsilon(x))$ belongs to $C_0(E_\varepsilon)$, if
\begin{equation}\label{kx-bis}
f(x, k_\varepsilon(x)) \ = \ \left\{
\begin{array}{l}
f (x,\star) = F_1(x), \; \; \mbox{ if } \;  x \in \varepsilon Y^{\sharp}, \\
f(x, \big\{\frac{x}{\varepsilon} \big\}) = F_2(x, \eta)|_{\eta = \{\frac{x}{\varepsilon}\}}, \;  \mbox{ if } \;  x \in \varepsilon \overline{G}^{\sharp}
\end{array}
\right.
\end{equation}
for some $F_1 \in C_0(\mathbb{R}^d), \ F_2 \in C_0(\mathbb{R}^d, \, C(\overline{G}))$.
Thus the space $C_0(E_\varepsilon)$ contains functions vanishing as $|x| \to \infty$ that are continuous separately on $ \varepsilon Y^{\sharp}$ and on any element of $\varepsilon \overline{G}^{\sharp}$. The norm in $C_0(E_\varepsilon)$ is defined as
\begin{equation}\label{normle}
\|f\|_{C_0(E_\varepsilon)} = 
 \sup\limits_{x \in \mathbb R^d}|f(x,k_\varepsilon(x))|.
\end{equation}



 Define the generator $\hat L_\varepsilon$ of the two-component random jump process $\hat X_\varepsilon(t) = (X_\varepsilon (t), k_\varepsilon(X_\varepsilon (t)))$ on $E_\varepsilon$ as follows
\begin{equation}\label{hatL-eps}
(\hat L_\varepsilon f)(x, k_\varepsilon(x)) = \frac{1}{\varepsilon^{d+2}} \int_{\mathbb R^d } p_\varepsilon (x,y) \big( f(y , k_\eps (y)) - f(x , k_\eps (x))\big) dy, \quad f \in C_0(E_\varepsilon)
\end{equation}
with the same transition rates  $p_\varepsilon (x,y)= a_\varepsilon(x-y) \Lambda_\varepsilon(x, y)$
as in (\ref{Leps}) for the operator $A_\eps$.
The operator $\hat L_\varepsilon$ is dissipative in $C_0(E_\varepsilon)$, i.e.
$$
\| \lambda f - \hat L_\varepsilon f \|_{C_0(E_\varepsilon)} \ge \lambda \| f \|_{C_0(E_\varepsilon)}, \quad \lambda >0.
$$
Then we conclude that $\hat L_\varepsilon$ is the generator of a strongly continuous contraction semigroup $T_\varepsilon(t)$ on  $C_0(E_\varepsilon)$:
$$
\| T_\varepsilon(t) f \|_{C_0(E_\varepsilon)} = \sup_{x \in \mathbb{R}^d } |T_\varepsilon (t) f (x, k_\eps (x))| \le \sup_{x \in \mathbb{R}^d } |f (x, k_\eps (x))|,  \quad f \in C_0(E_\varepsilon).
$$
\begin{remark}\label{RT}
Since the point  $(x,k_\eps (x)) \in E_\varepsilon$ is uniquely defined by its first coordinate  $x \in \mathbb R^d$, then we can use  $x \in \mathbb R^d$ as a coordinate in $E_\varepsilon$ (considering $E_\varepsilon$ as a graph of the mapping $k_\eps: \mathbb R^d \to G^\star $). In particular, this implies that the transition rates of the random jump on $E_\varepsilon$ coincide with the analogous ones on  $\mathbb R^d$ given in \eqref{Leps}.
\end{remark}
\bigskip

We proceed to constructing the limit semigroup.  We let $E= \mathbb R^d \times G^\star$, and $ C_0 (E)$ stands for the Banach space of continuous functions vanishing at infinity. Any function $F=F(x,k) \in C_0(E)$ can be represented as a vector function
\begin{equation}\label{Fxk}
F(x,k) = \left(
\begin{array}{l}
f_0(x) \\
g(x, \xi)
\end{array}
\right)
\end{equation}
with $f_0(x) \in C_0(\mathbb R^d), \; g(x,\xi) \in C_0(\mathbb R^d, \, C(\overline{G}))$.
The first component $f_0(x)$ of $F(x,k)$ corresponds to $k=\star$ while  $F(x,k) = g(x, \xi)$ if $k=\xi \in \overline{G}$.
In what follows we use the notation $\xi$ for the second coordinate of $g(x, k)$ when $k \in \overline{G}$ to stress that
$\xi$ is a coordinate on the period $\mathbb{T}^d$.

The norm in $C_0(E)$ is given by
$$
\|F\|_{C_0(E)} = \max \big \{ \max_{x\in \mathbb R^d} |f_0(x)|, \; \max_{x \in \mathbb R^d; \ \xi \in \overline{G}}  |g(x, \xi)| \big\}.
$$

Consider the operator
\begin{equation}\label{L-limit}
L F(x,k) =
\Theta \cdot \nabla \nabla f_0 (x) {\bf 1}_{\{k=\star\}}
 +  L_G F(x,k),
\end{equation}
where ${\bf 1}_{\{k=\star\}}$ is the indicator function,  $\Theta$ is  a positive definite matrix that will be defined below,
$\Theta \cdot \nabla \nabla f_0=\mathrm{Tr}(\Theta \nabla \nabla f_0 )$,
and  $L_G$ has the following form
\begin{equation}\label{LG}
L_G \left(
\begin{array}{l}
f_0 \\
g
\end{array}
\right) (x, \xi)  =
\left(
\begin{array}{l}
\int\limits_G \tilde b(\xi)(g(x, \xi) - f_0(x)) d\xi \\[3mm]
\int\limits_G \tilde a(\xi - \xi') v(\xi,\xi')(g(x, \xi') - g(x, \xi)) d\xi' + \tilde c(\xi)(f_0(x) - g(x, \xi))
\end{array}
\right)
\end{equation}
with
\begin{equation}\label{tilde-adef}
\tilde a(\xi) = \sum_{n \in \mathbb{Z}^d} a(\xi + n), \;\; \xi \in \mathbb{T}^d,
\end{equation}
and
\begin{equation}\label{b-c}
\begin{array}{c}
\displaystyle
\tilde b(\xi) \ = \ \frac{1}{|Y|} \ \int\limits_{Y} \tilde a(\eta - \xi) v(\eta,\xi) d \eta, \; \; \xi \in \overline{G}\\[1mm]
\displaystyle
\tilde c(\xi) \ = \ \int\limits_{Y} \tilde a(\xi -\eta) v(\xi,\eta) d \eta, \; \; \xi \in \overline{G};
\end{array}
\end{equation}
here $v(\eta,\xi)$ is the function defined above in \eqref{v-eps} - \eqref{vv}, and estimate \eqref{vv} implies that
\begin{equation}\label{b-c-est}
0 \le \tilde b(\xi)  \le b_2, \quad 0  \le \tilde c(\xi)  \le c_2,  \quad \mbox{ for all } \; \xi \in \overline{G},
\end{equation}
with some constants $b_2, \, c_2$.

Notice that $L_G$ is a bounded linear operator in $C_0(E)$ and it is a generator of a Markov jump process on $G^\star$.

\begin{lemma}\label{tilde}
Functions $\tilde b(\xi)$ and $\tilde c(\xi)$ are continuous: $\tilde b, \, \tilde c \in C(\overline{G})$.
\end{lemma}

\begin{proof}
The proof follows from condition \eqref{w} on $w(\xi, \eta)$. It is clear that both functions can be studied in the same way. For an arbitrary $\delta\in\mathbb R^d$  let us consider the increments
$
\triangle_\delta = |\tilde c(\xi + \delta)-\tilde c(\xi)|
$ of the function $\tilde c$.
After a simple rearrangement we get
\begin{equation}\label{lemma-tilde}
\begin{array}{c}
\displaystyle
\triangle_\delta =   \Big| \int\limits_{Y} \tilde a(\xi - \eta) w(\xi, \eta) d \eta - \int\limits_{Y-\delta} \tilde a(\xi -\eta) w(\xi+\delta,\eta+\delta) d \eta \Big|
\\[2mm]
\displaystyle
= \Big| \int\limits_{Y\cap (Y-\delta)} \tilde a(\xi - \eta) \big( w(\xi, \eta) - w(\xi+\delta, \eta+\delta) \big) d \eta
\\[2mm] \displaystyle
+ \int\limits_{Y\setminus (Y-\delta)} \tilde a(\xi -\eta) w(\xi,\eta) d \eta -
\int\limits_{(Y-\delta)\setminus Y} \tilde a(\xi -\eta) w(\xi+\delta,\eta+\delta) d \eta \Big|
\\[2mm]
\displaystyle
\le \varphi(\delta) \, \|\tilde a\|_{L^1 (\mathbb{T}^d)} + \|w\|_{C(\mathbb T^d)}\int\limits_{W_\delta} \tilde a(\xi -\eta)\, d \eta,
\end{array}
\end{equation}
where $\varphi(\delta)=\max\limits{\xi,\eta\in\mathbb T^d}|w(\xi+\delta,\eta+\delta)-w(\xi,\eta)|$, and
$W_\delta = Y \triangle (Y-\delta)$ is the symmetric difference of $Y$ and $Y\setminus\delta$.
Due to \eqref{w} we have $\varphi(\delta) \to 0$ as $\delta \to 0$.
Since $ \tilde a(\xi -\eta)  \in L^1(\mathbb{T}^d)$, 
and $\mbox{meas } W_\delta \to 0$ as $\delta \to 0$, we conclude that both terms on the right-hand side of \eqref{lemma-tilde} tend to 0 as $\delta \to 0$.
\end{proof}

Let $D(L)\subset C_0(E)$ be the domain of the operator $L$.

\begin{lemma}\label{HI}
The closure of the operator $(L, D(L))$ is the generator of a strongly continuous 
contraction semigroup on $C_0(E)$.
\end{lemma}

\begin{proof}
We check now that all conditions of the Hille-Yosida theorem are fulfilled for the operator \eqref{L-limit}.
The set of functions
\begin{equation}\label{core}
D = \{ h_0(x) \in S(\mathbb R^d), \;  h_1(x, \xi) \in S (\mathbb R^d, \, C(\overline{G})) \},
\end{equation}
is dense in $ C_0 (E)$ and $D \subset D(L)$, consequently $ D(L)$ is dense in  $ C_0 (E)$.

It is easy to see that the operator $L$ satisfies the positive maximum principle, i.e. if
$
F \in D(L)$  and  $\max_{ E } F(x,k) =  F(x_0, k_0),
$
then  $L F (x_0, k_0) \le 0$. Indeed, from (\ref{L-limit}) - (\ref{LG}) we obtain
$$
L F(x_0, \star ) = \Theta \cdot \nabla \nabla f_0 (x_0) + L_G F(x_0, \star) \le 0  \;   \mbox{ if } \; \max_{ E } F(x,k) = f_0(x_0) \; \big(\mbox{i.e. } (x_0, k_0) = (x_0, \star) \big),
$$
or
$$
L F(x_0, k) = L_G F(x_0, k) \le 0 \;    \mbox{ if } \; \max_{ E } F(x,k) = g(x_0, \xi_0) \; \big( \mbox{i.e. } (x_0, k_0) = (x_0, \xi_0), \ \xi_0 \in G \big).
$$
Thus, $L$ is dissipative.

Let us prove that  ${\cal R} (\lambda - L)$ is dense in $C_0 (E)$ for some $\lambda>0$. Rewrite the equation $(\lambda- L) F = H$ with $H = (h_0, h_1) \in D$ as follows:
\begin{equation}\label{LFH}
\begin{array}{c}
\displaystyle
\lambda f_0 (x) -\Theta \cdot \nabla \nabla f_0 (x)  - \int\limits_G \tilde b(\xi)(g(x, \xi) - f_0(x)) d\xi = h_0(x), \\[3mm]
\displaystyle
\lambda g(x, \xi) -\int\limits_G \tilde a(\xi - \xi') v(\xi,\xi')(g(x, \xi') - g(x, \xi)) d\xi' - \tilde c(\xi)(f_0(x) - g(x, \xi)) = h_1(x,\xi).
\end{array}
\end{equation}

The second equation in \eqref{LFH} can be rewritten as
\begin{equation}\label{LFH-1}
  -\int\limits_G \tilde a(\xi - \xi') v(\xi,\xi')(g(x, \xi') - g(x, \xi)) d\xi' + \big(\tilde c(\xi) + \lambda \big) g(x, \xi) = \tilde c(\xi)f_0(x) + h_1(x,\xi).
\end{equation}
Clearly, the operator on the left-hand side of  \eqref{LFH-1} is invertible for sufficiently large $\lambda>0$. Since equation \eqref{LFH-1} is linear,  its solution is the sum  $ g^\lambda(x, \xi) =  g^\lambda_1(x, \xi) + g^\lambda_2 (x, \xi)$ of solutions to two equations with the right-hand sides equal to $\tilde c(\xi) f_0(x)$ and $h_1(x,\xi)$ respectively. Then these solutions
$$
g^\lambda_1(x,\xi) = f_0(x) \, \varphi^\lambda(\xi) \in S (\mathbb R^d, \, C(\overline{G})) \; \mbox{ and } \; g^\lambda_2 (x, \xi) \in S (\mathbb R^d, \, C(\overline{G}))
$$ 
are bounded for sufficiently large $\lambda>0$.
Function $\varphi^\lambda(\xi)$ is a solution of equation
$$
 -\int\limits_G \tilde a(\xi - \xi') v(\xi,\xi')(\varphi^\lambda(\xi') - \varphi^\lambda(\xi)) d\xi' + \big(\tilde c(\xi) + \lambda \big) \varphi^\lambda(\xi) = \tilde c(\xi),
$$
and moreover  $\|\varphi^\lambda\|_{C(\overline{G})} \to 0$ as $\lambda \to \infty$.
Substituting $ g^\lambda(x, \xi) =  f_0(x) \, \varphi^\lambda(\xi) + g^\lambda_2 (x, \xi)$ into the first equation in \eqref{LFH} we get
\begin{equation}\label{f0-lambda}
\Big(\lambda - \Theta \cdot \nabla \nabla  -  \int\limits_G \tilde b(\xi) (1- \varphi^\lambda(\xi)) d\xi \Big) f_0(x)   = h_0(x) + \int\limits_G \tilde b(\xi) g^\lambda_2(x, \xi) d\xi.
\end{equation}
Since
the operator on the left-hand side of \eqref{f0-lambda} is invertible for sufficiently large $\lambda$, equation \eqref{f0-lambda} has a solution
 $f_0(x) \in S(\mathbb R^d)$.

Thus, applying the Hille-Yosida theorem we conclude that the closure of $L$ is a generator of a strongly continuous contraction semigroup $T(t)$  on $ C_0 (E)$, that is a Feller semigroup.
\end{proof}

Notice that $D$ is dense in $ C_0 (E)$ and  ${\cal R} \big( \lambda - L)|_D \big)$ is dense in $C_0(E)$
for some $\lambda>0$. Therefore,
the subspace $D \subset \ D(L)$ introduced in \eqref{core} is a core for $L$.

\medskip
For any $F \in C_0(E)$ we  define 
\begin{equation}\label{1}
(\pi_\varepsilon  F) (x, k_\varepsilon(x)) \ = \ \left\{
\begin{array}{ll}
f_0 (x), & \mbox{if} \; x \in \varepsilon Y^{\sharp}, \; (\mbox{or } k_\varepsilon(x) =\star); \\
g(x,  \big\{\frac{x}{\varepsilon} \big\}), & \mbox{if} \; x \in \varepsilon \overline{G}^{\sharp}, \; (\mbox{or } k_\varepsilon(x)= \big\{\frac{x}{\varepsilon} \big\} \in \overline{G}).
\end{array}
\right.
\end{equation}
Then $\pi_\varepsilon F \in C_0(E_\eps)$,  and $\pi_\varepsilon$ is a bounded linear transformation $ \pi_\varepsilon: C_0(E) \to C_0(E_\varepsilon)$:
\begin{equation}\label{hatpi1}
\| \pi_\varepsilon F \|_{C_0(E_\varepsilon)} = \sup_{ (x, k_\varepsilon(x)) \in E_\varepsilon} |(\pi_\varepsilon F) (x, k_\varepsilon(x))| \le \|  F \|_{C_0(E)}, \quad \; \sup_{\varepsilon} \|  \pi_\varepsilon \| \le 1.
\end{equation}

Our next step is to prove the convergence of the corresponding semigroup.


\begin{theorem}\label{T1}
Let $T(t)$ be the strongly continuous, positive, contraction semigroup on $ C_0 (E)$ with generator $L$ defined by \eqref{L-limit}--\eqref{LG}, and for each $\varepsilon>0$, $T_\varepsilon(t)$ be the strongly continuous, positive, contraction semigroup on $C_0(E_\varepsilon)$ defined above by its generators  \eqref{hatL-eps}.

Then for every $F \in C_0(E)  $
\begin{equation}\label{M-astral}
T_\varepsilon (t) \pi_\varepsilon F \ \to \  T(t) F   \quad \mbox{for all} \quad t \ge 0
\end{equation}
as $\varepsilon \to 0$.
\end{theorem}

\begin{proof}
In view of (\ref{normle}) to prove (\ref{M-astral}) it suffices to show that
\begin{equation}\label{R1}
\|  T_\varepsilon (t) \pi_\varepsilon F   - \pi_\varepsilon T(t) F  \|_{C_0(E_\varepsilon)} = \sup_{x \in \mathbb R^d}\left| T_\varepsilon (t) \pi_\varepsilon F (x, k_\varepsilon(x))  -
\pi_\varepsilon  T(t) F (x, k_\varepsilon(x)) \right| \to 0 \quad \mbox{as} \; \varepsilon \to 0.
\end{equation}
The proof of  (\ref{R1}) relies on the following  approximation theorem
\cite[Theorem 6.1, Ch.1]{EK}.\\[3mm]
{\bf Theorem}\,\cite{EK}.
{\sl \
For $n=1,2,\ldots$, let $\{T_n (t)\}$ and $\{T(t) \}$ be strongly continuous contraction semigroup on Banach spaces  ${\cal L}_n$ and ${\cal L}$ with generators $A_n$ and $A$. Let $D$ be a core for $A$. Then the following are equivalent: \\

a) For each $f \in {\cal L} $, $T_n (t) \pi_n f \ \to \  T(t) f \; $  as $\varepsilon \to 0$ for all  $t \ge 0$. \\

b) For each $f \in D$, there exists $f_n \in {\cal L}_n$ for each $n \ge 1$ such that $f_n \to f$ and $A_n f_n \to Af$.
}

\bigskip
We write $f_n \to f$ if $f_n \in {\cal L}_n$, $f\in {\cal L}$  and $\lim_{n \to \infty} \|f_n - \pi_n f \| =0$.
According to this theorem  the semigroups convergence stated in item a) is equivalent to
the statement in item b) which is the subject of the next lemma.

\begin{lemma}\label{Fn}
Let $L$ and $\hat L_\varepsilon$  be the operators defined by  \eqref{hatL-eps}, \eqref{L-limit}--\eqref{LG}, and the core $D \subset C_0(E)$ of $L$ is given by \eqref{core}. 
Then for every $F \in D$, there exists $F_\varepsilon \in  C_0(E_\varepsilon)$ such that
\begin{equation}\label{F0}
 \| F_\varepsilon - \pi_\varepsilon F\|_{ C_0(E_\varepsilon)} \to 0
\end{equation}
and
\begin{equation}\label{F1}
\| \hat L_\varepsilon F_\varepsilon - \pi_\varepsilon LF\|_{ C_0(E_\varepsilon )} \to 0 \quad \mbox{as } \; \varepsilon \to 0.
\end{equation}
\end{lemma}

\begin{proof}

For the proof, we are going to construct functions $F_\varepsilon \in C_0(E_\varepsilon)$ for which \eqref{F0} - \eqref{F1} hold.
Let $F = (f_0, g) \in D$, then we define $F_\varepsilon \in C_0(E_\varepsilon)$ as follows
\begin{equation}\label{4}
F_\varepsilon (x, k_\varepsilon(x)) \ = \ \left\{
\begin{array}{lr}
\displaystyle
f_0 (x) \; + \; \varepsilon \, \varphi ( \big \{\frac{x}{\varepsilon} \big\})\nabla f_0(x) \;
+ \;
\varepsilon^2 \varkappa ( \big \{\frac{x}{\varepsilon} \big\}) \nabla \nabla f_0 (x)
\\[4mm]
\displaystyle
+ \; \varepsilon^2  K \big( x, \big\{\frac{x}{\varepsilon} \big\} \big), 
&  \mbox{if} \;  x \in \varepsilon Y^{\sharp}
\\[4mm] \displaystyle
g(x, \big \{\frac{x}{\varepsilon} \big\}),  & \mbox{if} \;  x \in \varepsilon \overline{G}^{\sharp}
\end{array}
\right.
\end{equation}
Here  $\varphi(\xi) \in \big( C(Y) \big)^d, \; \varkappa(\xi) \in  \big( C(Y) \big)^{d^2}$ are 
continuous functions,  $ K(x, \eta) \in C_0(\mathbb{R}^d, C(Y))$ is a continuous function in both variables vanishing as $|x| \to \infty$. Functions  $\varphi(\xi)$ and $ \varkappa(\xi)$ will be constructed below in Proposition \ref{P1}, and see Proposition \ref{P2} for the construction of $ K(x,\eta)$.

From (\ref{1}) and (\ref{4}) it immediately follows that $F_\eps = \pi_\varepsilon F + G_\eps$ with $\|G_\eps \|_{ C_0(E_\varepsilon)} \le C \eps$. Consequently,
$$
\sup_{x \in \mathbb R^d} |  F_\varepsilon (x, k_\varepsilon(x)) - \pi_\varepsilon F(x, k_\varepsilon(x))  | = \| F_\varepsilon - \pi_\varepsilon F\|_{ C_0(E_\varepsilon)} \to 0
$$ as $\varepsilon \to 0$. Thus convergence (\ref{F0}) is proved.
\medskip

To justify the limit relation in (\ref{F1}) we will consider separately the cases when either $x \in \varepsilon Y^{\sharp}$ or  $x \in \varepsilon \overline{G}^{\sharp}$, and choosing the proper functions $q(\cdot), \ \varphi(\cdot), \ \varkappa(\cdot)$ in \eqref{4} we will prove that
\begin{equation}\label{conv1}
 \sup_{x \in \varepsilon Y^\sharp } |\hat L_\varepsilon F_\varepsilon (x, k_\varepsilon(x)) -
 \pi_\varepsilon LF (x, k_\varepsilon(x))|   \to 0
\end{equation}
and
\begin{equation}\label{conv2}
\sup_{x \in \varepsilon G^\sharp } |\hat L_\varepsilon F_\varepsilon (x, k_\varepsilon(x)) -
\pi_\varepsilon LF (x, k_\varepsilon(x))|   \to 0
\end{equation}
as $\varepsilon \to 0$.

Representations (\ref{Leps-sum}) - (\ref{Veps}) and \eqref{Lambda-eps} - \eqref{v-eps} yeild
\begin{equation}\label{hat-L-sum}
\hat L_\varepsilon = \hat L^0_\varepsilon + \hat V_\varepsilon,
\end{equation}
where
\begin{equation}\label{hat-L0eps}
\big( \hat L^0_\varepsilon f \big) (x, k_\varepsilon(x)) = \frac{1}{\varepsilon^{d+2}} \int_{\mathbb{R}^d} a(\frac{x-y}{\varepsilon}) \, \Lambda^0(\frac{x}{\varepsilon}, \frac{y}{\varepsilon}) \, \big( f(y ,k_\varepsilon(y)) -
f(x ,k_\varepsilon(x))\big) dy, \;\; f \in C_0(E_\eps),
\end{equation}
is the generator of the random jump process on the perforated domain $\varepsilon Y^{\sharp} = \mathbb R^d \setminus \varepsilon \overline{G}^{\sharp}$, and the operator
\begin{equation}\label{hat-Veps}
\big( \hat V_\varepsilon f \big) (x, k(x) ) = \frac{1}{\varepsilon^{d}} \int_{\mathbb{R}^d} a(\frac{x-y}{\varepsilon}) \, v(\frac{x}{\varepsilon}, \frac{y}{\varepsilon})\, \big( f(y ,k_\varepsilon(y)) - f(x ,k_\varepsilon(x))\big) dy, \;\;  f \in C_0(E_\eps),
\end{equation}
can be seen as a "perturbation" of $\hat L^0_\varepsilon$.

Since the second component in $E_\varepsilon$ is a function of the first one, in the remaining part of the proof we will write $F_\varepsilon(x)$ instead of $F_\varepsilon(x,k_\varepsilon(x))$ for the function $F_\eps$ given by \eqref{4}.

First we prove the convergence relation \eqref{conv1}. For $x \in \varepsilon Y^{\sharp}$ the first component of $F_\varepsilon$ can be written, see \eqref{4}, as a sum
\begin{equation}\label{5}
F_\varepsilon (x) = F_\varepsilon^P(x) + F_\varepsilon^Q(x), \quad  x \in \varepsilon Y^{\sharp},
\end{equation}
where
\begin{equation}\label{FP}
F_\varepsilon^P (x) = f_0 (x) + \varepsilon \varphi ( \big \{\frac{x}{\varepsilon} \big\})\nabla f_0(x)
+\eps^2\varkappa ( \big \{\frac{x}{\varepsilon} \big\})\nabla\nabla f_0(x)
\end{equation}
\begin{equation}\label{FQ}
 F_\varepsilon^Q (x) = \varepsilon^2  K \big( x, \big\{\frac{x}{\varepsilon} \big\} \big).
\end{equation}
Then
\begin{equation}\label{Ldec}
\hat L_\varepsilon F_\varepsilon = (\hat L_\varepsilon^0 + \hat V_\varepsilon)F_\varepsilon = \hat L_\varepsilon^0 (F_\varepsilon^P  +  F_\varepsilon^Q) + \hat V_\varepsilon F_\varepsilon = \hat L_\varepsilon^0 F_\varepsilon^P  +  \hat L_\varepsilon^0 F_\varepsilon^Q  + \hat V_\varepsilon F_\varepsilon.
\end{equation}

\begin{proposition}\label{P1}
There exist  functions $\varphi( \eta) \in (C(Y))^d$ and $\varkappa(\eta) \in (C({Y}))^{d^2}$ and a positive definite matrix $\Theta>0$, such that
\begin{equation}\label{PP1}
\hat L_\varepsilon^0 F_\varepsilon^P  \ \to \Theta \cdot \nabla \nabla f_0, \quad \; \mbox{i.e. } \ \sup_{x \in \varepsilon Y^{\sharp}}|\hat L_\varepsilon^0 F_\varepsilon^P (x) -  \Theta \cdot \nabla \nabla f_0(x)| \to 0  \quad \mbox{as } \; \varepsilon \to 0.
\end{equation}
\end{proposition}

The proof of this proposition is based on the corrector techniques, it is given in the Appendix.
\bigskip

The last Proposition allows us to pass to the limit in the first term on the right-hand side of \eqref{Ldec}. We now turn to
rearranging the other two terms.
Using \eqref{hat-L0eps}, \eqref{hat-Veps} together with \eqref{Lambda0} - \eqref{v-eps} and \eqref{FQ} we have for  $x \in \varepsilon Y^{\sharp}$
\begin{equation}\label{Y-1}
\begin{array}{l}
\displaystyle
(\hat L_\varepsilon^0 F_\varepsilon^Q + \hat V_\varepsilon F_\varepsilon)(x)  =  \frac{1}{\varepsilon^{d+2}} \int\limits_{\varepsilon Y^\sharp} a(\frac{x-y}{\varepsilon}) \, \Lambda^0(\frac{x}{\varepsilon}, \frac{y}{\varepsilon}) \, \big( F_\varepsilon^Q(y) - F_\varepsilon^Q(x) \big) dy
\\[5mm] \displaystyle
   + \frac{1}{\varepsilon^{d}} \int\limits_{\varepsilon G^\sharp} a(\frac{x-y}{\varepsilon}) \, v(\frac{x}{\varepsilon}, \frac{y}{\varepsilon})\, \big( g(y , \{ \frac{y}{\varepsilon} \}) - f_0(x) \big) dy
   + O(\eps)
\\[5mm] \displaystyle
=  \frac{1}{\varepsilon^{d}} \int\limits_{\varepsilon Y^\sharp} a(\frac{x-y}{\varepsilon}) \, \Lambda^0(\frac{x}{\varepsilon}, \frac{y}{\varepsilon}) \Big( K \big( y, \{\frac{y}{\varepsilon} \} \big) -
K \big( x, \{\frac{x}{\varepsilon} \} \big) \Big)  \, dy
 \\[5mm] \displaystyle
 \;  + \frac{1}{\varepsilon^{d}} \int\limits_{\varepsilon G^\sharp} a(\frac{x-y}{\varepsilon}) \, v(\frac{x}{\varepsilon}, \frac{y}{\varepsilon})\, \big( g(y , \{ \frac{y}{\varepsilon} \}) - f_0(x) \big) dy
 + O(\eps)
\end{array}
\end{equation}

Making change variables $z=\frac{x-y}{\varepsilon}$, using \eqref{Lambda0} - \eqref{v-eps} and taking into account the continuity and periodicity of $ K (x, \frac{x}{\varepsilon} )$ in the second variable we can rewrite \eqref{Y-1} for  $x \in \varepsilon Y^{\sharp}$ as follows:
\begin{equation}\label{Y-2bis}
\begin{array}{l}
\displaystyle
(\hat L_\varepsilon^0 F_\varepsilon^Q + \hat V_\varepsilon F_\varepsilon)(x) =
\\[5mm] \displaystyle
\int\limits_{\mathbb{R}^d} a(z) \, \Lambda^0(\frac{x}{\varepsilon}, \frac{x}{\varepsilon} - z) \big( K ( x - \varepsilon z, \frac{x}{\varepsilon} - z ) -  K (x, \frac{x}{\varepsilon} ) \big) d z
\\[5mm] \displaystyle
+ \int\limits_{\mathbb{R}^d} a(z) \, v(\frac{x}{\varepsilon}, \frac{x}{\varepsilon} - z) ( g(x- \varepsilon z, \frac{x}{\varepsilon} - z) -  f_0(x) ) \, dz + O(\varepsilon)
 \\[5mm] \displaystyle
= \int\limits_{\mathbb{R}^d} a(z) \, \Lambda^0(\frac{x}{\varepsilon}, \frac{x}{\varepsilon} - z) \big( K ( x - \varepsilon z, \frac{x}{\varepsilon} - z ) -  K (x, \frac{x}{\varepsilon} -z ) \big) d z
 \\[5mm] \displaystyle
 +\int\limits_{\mathbb{R}^d} a(z) \, \Lambda^0(\frac{x}{\varepsilon}, \frac{x}{\varepsilon} - z) \big( K ( x, \frac{x}{\varepsilon} - z ) -  K (x, \frac{x}{\varepsilon} ) \big) d z
   \\[5mm] \displaystyle
 + \int\limits_{\mathbb{R}^d} a(z) \, v(\frac{x}{\varepsilon}, \frac{x}{\varepsilon} -z )\, ( g(x - \varepsilon z, \frac{x}{\varepsilon} -z ) -  g(x, \frac{x}{\varepsilon} -z ) ) d z
  \\[5mm] \displaystyle
 + \int\limits_{\mathbb{R}^d} a(z) \, v(\frac{x}{\varepsilon}, \frac{x}{\varepsilon} -z )\,( g(x, \frac{x}{\varepsilon} -z ) -  f_0(x) ) d z + O(\eps).
\end{array}
\end{equation}
Since  $(f_0, g) \in D$ and $ K(x, \eta) \in C_0(\mathbb{R}^d, C(Y))$, then
\begin{equation}\label{y-x}
|K(x-\varepsilon z, \eta) -  K(x, \eta)| \le C(f,g) \varepsilon |z|, \quad | g(x-\varepsilon z, \xi) -  g(x, \xi) | \le C(g) \varepsilon |z| \quad \mbox{for all } \; \xi \in G,
\end{equation}
where $C(f,g),\ C(g)$ are constants depending only on functions $f$ and $g$ respectively.

We get using condition \eqref{a-1} on the kernel $a(z)$ 
that the first and the third integrals in \eqref{Y-2bis} are of order $\varepsilon$, and consequently
tend to 0 as $\varepsilon \to 0$.
Thus we obtain
\begin{equation}\label{Y-3}
(\hat L_\varepsilon^0 F_\varepsilon^Q + \hat V_\varepsilon F_\varepsilon)(x) =  - {\cal A}^0 K (x, \frac{x}{\varepsilon}) + \int\limits_G  \tilde a(\frac{x}{\varepsilon}- \xi) \, v(\frac{x}{\varepsilon}, \xi) ( g(x,\xi) - f_0(x) \big) d\xi   + O(\varepsilon),
\end{equation}
where  $\tilde a$ is given by \eqref{tilde-adef}
and
\begin{equation}\label{calA}
 {\cal A}^0 K (x, \eta)  = \int\limits_{Y} \tilde a(\eta-\eta') \, \Lambda^0(\eta, \eta') ( K (x, \eta' ) -  K (x,\eta ) ) d \eta', \quad \eta \in Y.
\end{equation}
Although we only consider the case $\Lambda^0(\eta, \eta') = 1$, see \eqref{Lambda0},  all further reasoning remains valid for any symmetric continuous function $\Lambda^0(\eta, \eta') = \Lambda^0(\eta', \eta)$ satisfying estimate \eqref{vv}.

Remind that $\tilde b(\xi)$ is defined in \eqref{b-c} as
\begin{equation*}\label{tilde-b}
\tilde b(\xi) =  \frac{1}{|Y|} \int\limits_Y \tilde a(\eta-\xi) \, v (\eta, \xi ) d \eta, \quad \xi \in G.
\end{equation*}

\begin{proposition}\label{P2}
There exists $K(x,\eta) \in  C_0(\mathbb R^d, \, C(Y) ), $ such that
\begin{equation}\label{7}
\sup_{x \in \varepsilon Y^{\sharp}} \left| (\hat L_\varepsilon^0 F_\varepsilon^Q + \hat V_\varepsilon F_\varepsilon)(x) - \int\limits_G \tilde b(\xi) \big( g(x,\xi) - f_0(x) \big) d\xi  \right| \ \to \ 0  \quad \mbox{as } \; \varepsilon \to 0,
\end{equation}
for all $f_0(x) \in C_0(\mathbb R^d), \; g(x,\xi) \in C_0(\mathbb R^d, \, C(\overline{G}))$.
\end{proposition}

\begin{proof}
According to \eqref{Y-3} relation  \eqref{7} holds, if the function  $ K(x,\eta)$ satisfies the equation
\begin{equation}\label{P2-0}
 - {\cal A}^0 K (x,\eta) + \int\limits_G  \tilde a(\eta - \xi) \, v(\eta, \xi) \big( g(x,\xi) - f_0(x) \big) d\xi = \int\limits_G \tilde b(\xi) \big( g(x,\xi) - f_0(x) \big) d\xi,
\end{equation}
or equivalently,
\begin{equation}\label{P2-3}
{\cal A}^0 K (x,\eta)  = \Psi(x, \eta)
\end{equation}
with
\begin{equation}\label{P2-2}
\Psi(x, \eta)= \int\limits_G  \tilde a(\eta - \xi) v(\eta, \xi) ( g(x,\xi) - f_0(x)) d\xi - \int\limits_G \tilde b(\xi) ( g(x,\xi) - f_0(x) ) d\xi.
\end{equation}
Since $\tilde a \in L^1(\mathbb{T}^d)$ and $v(\eta, \xi), \tilde b(\xi)$ are continuous functions, then $\Psi(x, \eta)$ is continuous in $\eta$. Moreover, the properties of $ g(x,\xi), f_0(x)$ imply that $\Psi(x, \eta) \in  C_0(\mathbb R^d, \, C(Y) )$ and in particular  $\Psi(x, \eta) \in  L^\infty (\mathbb R^d \times Y )$.
Thus, we need to prove that the solution of \eqref{P2-3} has the same properties as the function $\Psi(x,\eta)$ on the right side of the equation.

Due to \eqref{calA} we can rewrite equation \eqref{P2-3} as follows
\begin{equation}\label{P2-5}
\int\limits_{Y} \tilde a(\eta-\eta') K(x,\eta') d \eta' -  K(x,\eta) \, \int\limits_{Y} \tilde a(\eta-\eta') d \eta' = \Psi(x, \eta).
\end{equation}
Let us notice that conditions \eqref{a-1} - \eqref{a-1_low} on the function $a(\cdot)$ imply that
\begin{equation}\label{P2-4}
p(\eta) = \int\limits_{Y} \tilde a(\eta-\eta') d \eta' \ge c_0>0 \; \; \mbox{ for all } \; \eta \in Y \qquad \mbox{and } \quad p(\cdot) \in C(Y).
\end{equation}
Then using \eqref{b-c} and \eqref{P2-2} we conclude that the solvability condition
$\int\limits_Y \Psi(x, \eta) d \eta = 0$ holds, i. e. $\Psi(x, \eta)\, \bot \, {\bf 1}_Y$,  where
${\bf 1}_Y = {\rm Ker} \, \big( {\cal A}^0 \big)^\ast$. Consequently, there exists a solution $K(x, \cdot) \in L^2(Y)$ of  for every $x \in \mathbb{R}^d$.
Moreover, Theorem 5.4 from \cite{PSSZ} implies that $K(x, \eta) \in L^\infty(\mathbb{R}^d \times Y)$.
Since $K(x,\eta)$ is a bounded function, then using the same arguments as in Lemma \ref{tilde} we conclude that $\int\limits_{Y} \tilde a(\eta-\eta') K(x,\eta') d \eta'$ is continuous in $\eta$. Thus \eqref{P2-5} - \eqref{P2-4} yield that  $K(x,\eta) \in C_0(\mathbb{R}^d, C(Y))$.
\end{proof}

Thus taking into account \eqref{LG} and combining (\ref{PP1}) and (\ref{7}) we get \eqref{conv1}.
\medskip

To obtain \eqref{conv2} next we explore the case when $x \in \varepsilon G^{\sharp}$.
Due to \eqref{4}, \eqref{hat-Veps} and \eqref{Ldec} we can write
\begin{equation}\label{8}
\begin{array}{l}
\displaystyle
\hat L_\varepsilon F_\varepsilon(x) = (\hat L_\varepsilon^0 + \hat V_\varepsilon)F_\varepsilon(x) =  \hat V_\varepsilon F_\varepsilon(x)
=  \frac{1}{\varepsilon^d} \int\limits_{\varepsilon G^\sharp} a (\frac{x-y}{\varepsilon}) v (\frac{x}{\varepsilon}, \frac{y}{\varepsilon}) (g (y, \frac{y}{\varepsilon}) - g (x, \frac{x}{\varepsilon})) dy
\\[5mm] \displaystyle
+ \ \frac{1}{\varepsilon^d} \int\limits_{\varepsilon Y^\sharp}  a (\frac{x-y}{\varepsilon}) v (\frac{x}{\varepsilon}, \frac{y}{\varepsilon}) (f_0 (y) + O(\varepsilon) - g(x, \frac{x}{\varepsilon})) dy.
\end{array}
\end{equation}
In the same way as above using estimates \eqref{y-x} we have for $x \in \varepsilon G^\sharp$, $\xi = \big\{ \frac{x}{\varepsilon}\big\} \in G$
\begin{equation}\label{8bis}
\hat L_\varepsilon F_\varepsilon(x) = \hat V_\varepsilon F_\varepsilon(x) =
  \int\limits_{ G} \tilde a(\xi - \xi') \, v(\xi, \xi')\, \big( g(x , \xi') - g(x, \xi) \big) d\xi' + \tilde c(\xi) (f_0 (x) - g(x, \xi)) +  O(\varepsilon),
\end{equation}
where
\begin{equation}\label{tilde-c}
\tilde c(\xi) =  \int\limits_Y \tilde a(\xi - \eta) \, v (\xi, \eta ) d \eta, \quad \xi \in G.
\end{equation}

Finally, (\ref{conv2}) is a consequence of \eqref{LG} and (\ref{8bis}) - (\ref{tilde-c}). Lemma \ref{Fn} is proved.
\end{proof}

It remains to recall that (\ref{R1}) is a straightforward consequence of the above approximation theorem. This completes the proof of Theorem \ref{T1}.
\end{proof}


\section{Invariance principle, convergence of the processes}
\label{sec3}

In the previous section we justified the convergence of the semigroups, and consequently, the finite dimensional distributions of  $ \hat X_\varepsilon(t)$. In this section we construct the limit Markov process on $E$ with generator $L$ given by (\ref{L-limit}) - (\ref{LG}) using approximation theorems from \cite{EK}.

Let us describe first the process $\hat X_\varepsilon (t)$.
For the two-component random jump process  $\hat X_\varepsilon (t) = ( X_\varepsilon (t),  k(X_\varepsilon (t)) ) $ on $E_\varepsilon$ given by the generator (\ref{hatL-eps}), the second component $ k(X_\varepsilon (t)) \in G^\star$ is the function of the first component $ X_\varepsilon (t) \in  \mathbb R^d$ and the mapping $k_\varepsilon$. Thus Markov processes $\hat X_\varepsilon (t)$ and $X_\varepsilon (t)$ are equivalent, i.e. the trajectories of $\{ \hat X_\varepsilon(t) \}$ are isomorphic to trajectories of  $\{ X_\varepsilon(t) \}$. However, the second component of  $X_\varepsilon(t)$ plays the crucial role when passing to the limit $\varepsilon \to 0$. As has been shown in Section 2 the limit process $\mathcal X(t)$ preserves the Markov property only in the presence of the second component $k(t)$. It should be noted that in the process $\mathcal X(t)$ the second component is not a function of the first one anymore. This can be shown by analysing  the structure of the limit generator, see (\ref{L-limit}) - (\ref{LG}).

The goal of this section is to prove the existence of the limit process $\mathcal X(t)$ on $E$ corresponding to the Feller semigroup $T(t)$ with sample paths in $D_E [0, \infty)$ and to establish the invariance principle for the processes $\hat X_\varepsilon (t)$. Namely, we show that $\hat X_\varepsilon (t)$ converges in distributions, as $\varepsilon \to 0$, to $\mathcal X(t)$ in the Skorokhod topology of $D_E [0, \infty)$.

\begin{theorem}\label{T2}
For any initial distribution $\nu \in {\cal P}(E)$ there exists a Markov process $\mathcal X(t)$ corresponding to the Feller semigroup $T(t): C_0(E) \to C_0(E)$ with generator $L$ defined by \eqref{L-limit} -- \eqref{LG} and with sample paths in $D_E [0, \infty)$.

If $\nu$ is the limit law of $\hat X_\varepsilon(0)$ as $\varepsilon \to 0$, then
\begin{equation}\label{TT2}
\hat X_\varepsilon (t) \ \Rightarrow \ \mathcal X(t) \quad \mbox{ in } \; D_E [0, \infty) \; \mbox{ as } \; \varepsilon \to 0.
\end{equation}
\end{theorem}

\begin{proof}
The main idea of the proof is to combine the convergence of the finite dimensional distributions of $\hat X_\varepsilon(t)$ (that is a consequence of Theorem \ref{T1} see Remark \ref{R2}) and the tightness of $\hat X_\varepsilon(t)$ in $D_E [0, \infty)$.

We apply here Theorem 2.11 from \cite{EK}, Chapter 4. For the reader convenience we formulate it here.

\bigskip\noindent
{\bf Theorem} \cite{EK}.{\sl \
Let $E, E_1, E_2, \ldots $ be metric spaces with $E$ locally compact and separable. For $n=1,2, \ldots$ let $\eta_n:E_n \to E$ be measurable, let $\{ T_n(t) \}$ be a semigroup on $B (E_n)$ (a Banach space of bounded functions with sup-norm) given by a transition function, and suppose $Y_n$ is a Markov process in $E_n$ corresponding to  $\{ T_n(t) \}$, such that $X_n = \eta_n \circ Y_n$ has sample paths in  $D_E [0, \infty)$. Define $\pi_n: B(E) \to  B (E_n)$ by $\pi_n f = f \circ \eta_n$. Suppose that $\{T(t)\}$ is a Feller semigroup on $C_0 (E)$ and that for each $f \in C_0(E)$ and $t \ge 0$
\begin{equation}\label{EK1}
\| T_n(t) \pi_n f \ - \ \pi_n T(t) f \| \ \to \ 0.
\end{equation}
If $\{ X_n (0) \}$ has limiting distribution $\nu \in {\cal P} (E)$, then there is a Markov process $X$ corresponding to $\{ T(t) \}$ with initial distribution $\nu$ and sample paths in $D_E [0, \infty)$, and $X_n \Rightarrow X$. }

\bigskip
In our case, $E=\mathbb R^d \times G^\star$, $E_n=E_\varepsilon \subset E, \; \varepsilon = \frac1n$, and $\eta_n = \eta_\varepsilon: \ E_\varepsilon \to E$ is the measurable mapping for every $\varepsilon$, it is embedding of the set $E_\varepsilon$ to $E$.
The Markov process $Y_n(t) = Y_\varepsilon(t)$ is the same as the random jump process $(X_\varepsilon (t)$ with transition rates $p_\varepsilon (x,y)$, and $\hat X_\varepsilon (t) =  \eta_\varepsilon \circ Y_\varepsilon =  (X_\varepsilon (t),  k(X_\varepsilon (t)) )$ is the two-component random jump process on $E_\varepsilon$ (see (\ref{hatL-eps})).
The semigroup $T(t)$ on $C_0(E)$ generated by the operator $L$, see  (\ref{L-limit}) - (\ref{LG}), is the Feller semigroup by the Hille-Yosida theorem as was mentioned in the beginning of Section 2.
The convergence in (\ref{EK1}) is the subject of Theorem \ref{T1}.

Thus, all assumptions of Theorem 2.11 from \cite{EK} are fulfilled. Consequently, the processes $\hat X_\varepsilon (t) =  \eta_\varepsilon \circ Y_\varepsilon$ convergence in law in the space $D_E[0, \infty)$. This completes the proof of Theorem \ref{T2}.
\end{proof}
\medskip

In the conclusion of this section we describe the limit process $\mathcal X(t)$ corresponding to the semigroup $T(t)$ with generator $L$ defined by \eqref{L-limit} -- \eqref{LG}. It is a two component process $\mathcal X(t) = \{ x(t), k(t) \}$ with the first component $x(t)$ taking value in $\mathbb{R}^d$ and the second one $k(t)$ from $G^\star$. Recall that $L$ is a sum of two operators, the first one is the generator of diffusion in $\mathbb{R}^d$, the second one is the generator of random jump process on $G^\star$. Moreover, as follows from  \eqref{L-limit} the generator of diffusion "turns on" only if $k = \star$. Assume that the initial distribution $\nu= \delta_0(x) \, \delta_\star (k)$, so that $\mathcal X(0) = (0, \star)$. The second component $k(t)$ of $\mathcal X(t)$ stays at the point $\star$ a random time and then jumps to an other point in $G^\star$. During these random waiting time  $k = \star$ and the first component $x(t)$ of $\mathcal X(t)$ evolves as a
diffusion generated by the operator $\Theta \cdot \nabla\nabla$. Then $x(t)$ immediately stops moving when the second component $k(t)$ takes values in $G$ (not $\star$). After that, the second component $k(t)$ evolves as a Markov jump process with states in $G^\star$, and after a random time it arrives again at the point  $ \ star $.
Then the first component $x(t)$ of $\mathcal X(t)$ 
"wakes up" and continues the diffusive motion from the place where it stopped earlier.

\section{Memory kernel in evolution of the first (spatial) component of $\mathcal X(t)$}
\label{sec4}

In this section we derive the evolution equation with a memory term for  the first component $x(t)$ of $\mathcal X(t)$.
We consider the parabolic equation
\begin{equation}\label{E1-1}
\frac{\partial}{\partial t} F(x,k,t) = L F(x,k,t), \quad
F(x,k,t) = \left(
\begin{array}{l}
f_0(x,t) \\
g(x, \xi,t)
\end{array}
\right)
\end{equation}
with operator $L$ defined in \eqref{L-limit} - \eqref{LG}.

Let us start with the second line of \eqref{E1-1}:
\begin{equation}\label{E1-2}
\begin{array}{l}
\frac{\partial}{\partial t} g(x,\xi,t) = \int\limits_G \tilde a(\xi - \xi') v(\xi,\xi')(g(x, \xi',t) - g(x, \xi,t)) d\xi' + \tilde c(\xi)(f_0(x,t) - g(x, \xi,t))\\[3mm]
= \mathcal{A} g(x, \xi, t) + \tilde c(\xi)f_0(x,t),
\end{array}
\end{equation}
where $\mathcal{A}$ is a linear bounded operator defined by
\begin{equation}\label{E1-3}
\mathcal{A} g(x, \xi, t) =  \int\limits_G \tilde a(\xi - \xi') v(\xi,\xi')(g(x, \xi',t) - g(x, \xi,t)) d\xi' - \tilde c(\xi) g(x, \xi,t),
\end{equation}
and $\tilde c(\xi)f_0(x,t)$ is the non-homogeneous time dependent term in \eqref{E1-2}.
Let us consider first the homogeneous part of problem \eqref{E1-2}:
$$\frac{\partial}{\partial t} g_0 = \mathcal{A} g_0, $$
 with $ \mathcal{A}$ defined by \eqref{E1-3}.
Note that $\mathcal{A}$ is the sum $\mathcal{A} = B + C$ of the generator $B$ of the Markov jump process $\xi(t)$ on $\overline{G}$ and the operator $C$ of multiplication by the potential $- \tilde c(\xi)$ with $\tilde c(\xi)$ defined in \eqref{b-c}.
Consequently, the solution of the homogeneous part of problem \eqref{E1-2}  can be written as
\begin{equation}\label{E1-4bis}
g_0(x, \xi, t) =  \int\limits_G U(\xi, \xi', t) \pi_1(x,\xi') d \xi', \qquad g_0(x,\xi,0) = \pi_1(x,\xi),
\end{equation}
where
\begin{equation}\label{E1-5bis}
U(\xi, \xi', t) = e^{t \mathcal{A} }(\xi, \xi') =   e^{t(B+C)}(\xi, \xi').
\end{equation}
It is worth noting that the function $U(\xi, \xi', t)$ is exponentially decaying in $t$, since the spectrum of the operator $ \mathcal{A}$ defined by \eqref{E1-3} has a spectral gap: $\sigma(\mathcal{A}) \subset [-r_2, -r_1], \ -r_1 <0$.

Using the Feynman-Kac formula we get
$$
g_0(x, \xi, t) =  \mathbb{E}_{\xi} \Big[ \pi_1(x, \xi(t))\, \exp \big\{ -\int\limits_0^t \tilde c(\xi(\tau)) d \tau \big \} \Big],
$$
where
the average $\mathbb{E}_{\xi}$ is taken over the distribution of trajectories of the process $\{\xi(\tau), \, \tau \in [0,t]\}$ generated by the operator $B$ under the condition that $\xi(0) = \xi$.

The solution of non-homogeneous problem \eqref{E1-2} can be found now by the Duhamel's formula:
\begin{equation}\label{E1-4}
g(x, \xi, t) =  \int\limits_G U(\xi, \xi', t) \pi_1(x,\xi') d \xi' + \int\limits_0^t f_0(x,s) \int\limits_G U(\xi, \xi', t-s) \tilde c(\xi') d \xi' \, ds,
\end{equation}
where $ U(\xi, \xi', t)$ is defined by \eqref{E1-5bis}.

Substituting relation \eqref{E1-4} for $g(x, \xi, t)$ into the first equation in \eqref{E1-1} we get the required evolution equation for the spatial component $x(t)$ of  $\mathcal X(t)$:
\begin{equation}\label{E1-5}
\begin{array}{l}\displaystyle
\frac{\partial}{\partial t} f_0(x, t) = \Theta \cdot \nabla \nabla f_0(x,t) - \lambda(0) f_0(x,t) +
 \int\limits_G  \int\limits_G \tilde b(\xi) U(\xi,\xi',t) \pi_1(x, \xi') d\xi' d \xi \\[3mm]
 \displaystyle
+ \int\limits_0^t \left[  \int\limits_G  \int\limits_G  U(\xi, \xi',t-s) \, \tilde c( \xi') \, \tilde b(\xi)d \xi'd \xi \right] f_0(x,s) ds,
\end{array}
\end{equation}
where $\lambda(0) =  \int\limits_G   \tilde b(\xi)d \xi$ and the memory kernel has the form
$$
K(t-s) =  \int\limits_G  \int\limits_G  U(\xi, \xi',t-s) \, \tilde c( \xi') \, \tilde b(\xi)d \xi'd \xi.
$$

\section{Semigroup in $L^2(E)$}\label{sec5}

In this Section we discuss the construction of the semigroup $T(t)$ with generator $L$ given by \eqref{L-limit}-\eqref{LG} in $L^2(E)$, and pay special attention to the spectral analysis of the generator of this semigroup.
Any function $F \in L^2(E)$ has the form
\begin{equation}\label{Fxk-bis}
F = \big( f_0(x), \  g(x, \xi) \big), \quad f_0 \in L^2(\mathbb R^d), \;  g(x,\xi) \in L^2(\mathbb R^d \times \overline{G} ).
\end{equation}
Considering the semigroup convergence in $L^2(E)$, we drop assumption \eqref{w} but additionally to \eqref{vv} assume that  $w(\xi, \eta) = w(\eta, \xi)$ is a symmetric function on $\mathbb{T}^d \times \mathbb{T}^d$.

The definition of the bounded linear mapping $ \pi_\varepsilon: L^2(E) \to L^2(E_\varepsilon)$ should be slightly modified in the same way as in \cite{AA}.
Namely, for $ x \in \varepsilon \overline{G}^{\sharp} $ we define $ \hat x = \varepsilon \left[  \frac{x}{\varepsilon} \right] \in \varepsilon \mathbb Z^d$, $\xi=\left\{ \frac{x}{\varepsilon} \right\} \in \overline{G}$, and set
\begin{equation}\label{pi_L2}
(\pi_\varepsilon F) (x,  k_\varepsilon(x)) \ = \ \left\{
\begin{array}{ll}
f_0 (x), \; & \mbox{if} \; x \in \varepsilon Y^{\sharp}, \; (\mbox{or } k_\varepsilon(x) =\star); \\[2mm]
\displaystyle
\frac{1}{\eps^d\,|G|}\int_{\eps G} g(\hat x+\eta, \xi)\,d\eta, \; & \mbox{if} \; x-\hat x \in \varepsilon \overline{G}.
\end{array}
\right.
\end{equation}


\begin{theorem}
Assume that $w(\xi, \eta) = w(\eta, \xi)$ and all conditions \eqref{a-1} - \eqref{a-1_low}, \eqref{Lambda-eps} - \eqref{vv} are satisfied.
Then for each $F \in L^2(E)$
$$T_\varepsilon(t) \pi_\varepsilon F \to T(t) F, \quad \forall \  t \ge 0 \quad \mbox{ as } \; \varepsilon \to 0
$$
Moreover, this convergence is uniform on bounded time intervals.
\end{theorem}

\begin{proof}
We use the same arguments as in the proof of Theorem \ref{T1}. The crucial step in our constructions is anzatz \eqref{4} that we use for the approximation sequence.
\end{proof}

We proceed with the description of the spectrum of the limit operator  $L$ given by \eqref{L-limit}-\eqref{LG}.
Denote by
\begin{equation}\label{PK}
\begin{array}{c}
(\Phi g)(\xi) = \Phi(\xi) g(\xi), \quad  \Phi(\xi) =  \int\limits_{\mathbb{T}^d} \tilde a(\xi - \xi') v(\xi,\xi') \ d\xi',
\quad g \in L^2(\overline{G}); \\[5mm]
(K g)(\xi) = \int\limits_G \tilde a(\xi - \xi') v(\xi,\xi')g(\xi') \ d\xi', \quad K(\xi, \xi')=  \tilde a(\xi - \xi') v(\xi,\xi'), \quad    K(\xi, \xi')= K(\xi', \xi).
\end{array}
\end{equation}
Our assumptions on functions $a$ and $v$ imply that
$$0 < \gamma_1 \le \Phi(\xi) \le \gamma_2 < \infty \quad \mbox{with } \; \gamma_1 = \alpha_1 \| a\|_{L^1}, \ \gamma_2= \alpha_2 \| a \|_{L^1},$$ where bounds $\alpha_1, \ \alpha_2$ are defined in \eqref{vv}.

\begin{theorem}
Spectrum of the operator $-L$ is the union: $\sigma(-L) = \sigma_1 \cup \sigma_2$, with
\begin{equation}\label{sp1}
{\color{red}\sigma_1 = \sigma (\Phi - K)} \quad \mbox{ and } \quad \sigma_2 = \{ \lambda \ge 0: \ 1 +   \big(\tilde b, R_\lambda(\Phi-K) {\bf 1}  \big) \ge 0 \},
\end{equation}
where $R_\lambda(\Phi-K) =  (\Phi - K - \lambda)^{-1}$ is the resolvent of the operator $\Phi-K$.
In particular, there exist $\lambda_1$ and $\lambda_2$, $\; 0< \lambda_1 < \gamma_1 \le \gamma_2 < \lambda_2$, such that
\begin{equation}\label{sp2}
[0, \lambda_1] \cup {\mbox{Ran}}\, \Phi \cup [\lambda_2, \infty) \ \subset \ \sigma_{cont}(-L).
\end{equation}
\end{theorem}

\begin{proof}
The equation $-L F = \lambda F$ with $F=\big( f_0(x), \ f_0(x) {\bf{1}}(\xi) + g_1(x, \xi) \big)$ takes the form
\begin{equation}\label{S1}
\left(
\begin{array}{l}
-\Theta \cdot \nabla \nabla f_0 (x) -  \int\limits_G \tilde b(\xi)g_1(x, \xi) d\xi \\
\int\limits_G \tilde a(\xi - \xi') v(\xi,\xi')\big( g_1(x, \xi) - g_1(x, \xi')\big) d\xi' + \tilde c(\xi) g_1(x, \xi)
\end{array}
\right)
 = \lambda
\left(
\begin{array}{l}
f_0(x) \\ f_0(x) \, {\bf{1}}(\xi) + g_1 (x, \xi)
\end{array}
\right)
\end{equation}
The second line in \eqref{S1} can be rewritten as
\begin{equation}\label{S2}
g_1(x, \xi) = \lambda f_0(x) (\Phi - K - \lambda)^{-1} {\bf{1}}(\xi).
\end{equation}
The elimination of $g_1(x, \xi)$ in the first equation in \eqref{S1} gives
\begin{equation}\label{S3}
\begin{array}{l}
-\Theta \cdot \nabla \nabla f_0 (x) = \lambda f_0(x) +   \lambda f_0(x) \int\limits_G \tilde b(\xi)(\Phi - K - \lambda)^{-1} {\bf{1}}(\xi) d\xi \\
=  \lambda f_0(x) \Big[ 1+   \big( \tilde b, \ R_\lambda(\Phi - K) {\bf{1}} \big)  \big].
\end{array}
\end{equation}
Since the spectrum of the operator $-\Theta \cdot \nabla \nabla$ fills up the positive half-line, we obtain that all
$\lambda$ such that
\begin{equation}\label{S4-bis}
\lambda \Big[ 1+  \int\limits_G \tilde b(\xi)(\Phi - K - \lambda)^{-1} {\bf{1}}(\xi) d\xi \Big] = \lambda \big[ 1+   \big( \tilde b, \ R_\lambda(\Phi - K) {\bf{1}} \big) \big] \ge 0
\end{equation}
belong to the spectrum of $-L$.

\begin{lemma}\label{SL}
The spectral radius of the operator $\Phi^{-1}\, K$ is less then 1: $r(\Phi^{-1}\, K)<1$.
\end{lemma}
\begin{proof}
Since $\Phi^{-1}\, K$ is a positive compact operator, there exists the maximal eigenvalue $\mu \in \mathbb{R}, \ \mu>0$ and the corresponding positive eigenfunction $g_0(\xi)$:
$$
\Phi^{-1}\, K \, g_0 = \mu \, g_0.
$$
This equality is equivalent to
\begin{equation}\label{S5}
(K g_0, g_0) = \mu (\Phi g_0, g_0) = \mu \int\limits_G \Phi(\xi) g_0^2(\xi) d \xi.
\end{equation}
Moreover, using the same reasoning as in Lemma 4.1 in \cite{PZ19} one can prove the following two-sided bound for the function  $g_0(\xi)$:
\begin{equation}\label{g0}
0< \beta_1 \le g_0(\xi) \le \beta_2 \quad \forall \, \xi \in \mathbb{T}^d.
\end{equation}
Let us estimate the left hand side of equation \eqref{S5}:
\begin{equation}\label{S6}
\begin{array}{c} \displaystyle
(K g_0, g_0) = \int\limits_G \int\limits_G  K(\xi, \xi') g_0(\xi') g_0(\xi) d \xi' d \xi \le \frac12 \int\limits_G \int\limits_G  K(\xi, \xi') (g_0^2(\xi')+ g_0^2(\xi)) d \xi' d \xi \\[2mm] \displaystyle
=  \int\limits_G \int\limits_G  K(\xi, \xi') d \xi'  g_0^2(\xi) d \xi <  \int\limits_G \Phi(\xi) g_0^2(\xi) d \xi.
\end{array}
\end{equation}
Here we used \eqref{PK}.
Thus relations \eqref{S5} - \eqref{S6} yield that $r(\Phi^{-1}\, K) = \mu <1$.
\end{proof}

It follows from Lemma \ref{SL} that  $r\big( (\Phi + \nu)^{-1}\, K \big) <1$ for any $\nu>0$. Consequently the Neumann series for the operator $\big( I - (\Phi+\nu)^{-1} K \big)^{-1}$ converges and this operator has a positive kernel. Thus, taking $\lambda = -\nu <0$ we obtain
$$
 \big( \tilde b, \ R_{\lambda}(\Phi - K) {\bf{1}} \big) =  \big( \tilde b, \ R_{-\nu}(\Phi - K) {\bf{1}} \big) =  \big( \tilde b, \ (I - (\Phi+\nu)^{-1} K )^{-1} (\Phi+\nu)^{-1} {\bf{1}} \big) \ge 0,
$$
and
$$
-\nu \big[ 1+  \big( \tilde b, \ R_{-\nu}(\Phi - K) {\bf{1}} \big) \big] <0.
$$
Together with \eqref{S4-bis} this implies that $\lambda<0$ does not belong to the spectrum  of $-L$, and only  $\lambda \ge 0$ such that
\begin{equation}\label{S4}
1+  \int\limits_G \tilde b(\xi)(\Phi - K - \lambda)^{-1} {\bf{1}}(\xi) d\xi = 1+   \big( \tilde b, \ R_\lambda(\Phi - K) {\bf{1}} \big) \ge 0
\end{equation}
belong to the spectrum of $-L$. Thus, \eqref{sp1} is proved.
\medskip

Next we proceed to the study of condition \eqref{S4}.
Since $\Phi - K$ is a bounded operator, we get
$$
\big( \tilde b, \ R_{\lambda}(\Phi - K) {\bf{1}} \big) \to 0 \quad \mbox{as } \;  \lambda \to + \infty.
$$
Thus, there exists $\lambda_2 > \gamma_2$ such that $ |( \tilde b, \ R_{\lambda}(\Phi - K) {\bf{1}} )|<1$ for all $\lambda \ge \lambda_2$, and condition \eqref{S4} yields that $[\lambda_2, + \infty) \subset \sigma (-L)$.

On the other hand,  using Lemma \ref{SL} for small $\lambda>0$ we conclude that there exists
 $\lambda_1>0$ such that for all $\lambda \in [0,\lambda_1]$ Neumann series for the operator $\big( I - \Phi^{-1} (K+\lambda)\big)^{-1}$ converges. Consequently,  this operator has a positive kernel, and
$$
\big( \tilde b, \ R_{\lambda}(\Phi - K) {\bf{1}} \big) =  \big( \tilde b, \ (I - \Phi^{-1} (K+\lambda) )^{-1} \Phi^{-1} {\bf{1}} \big)\ge 0 \quad \mbox{for all } \; \lambda \in [0, \lambda_1].
$$
Thus, $[0, \lambda_1] \subset \sigma (-L)$.

\end{proof}

\section*{Appendix: The proof of Proposition \ref{P1}}

In this section we prove Proposition \ref{P1}. We remind that $f_0$ is a Schwartz class function, the function
$F^P_\eps$ is defined in   \eqref{FP}, and $\Lambda_0(\xi, \eta)=1$  as $ \xi, \eta \in Y$, see \eqref{Lambda0}. Our goal is to determine  periodic bounded functions $\varphi$ and $\varkappa$ and a constant
matrix $\Theta$ so that relation \eqref{PP1} holds true.

First we construct $\varphi=\varphi(\xi), \, \xi \in Y,$ as a periodic solution of the following equation
\begin{equation}\label{efp}
  \int_Y \widetilde{a}(\xi-\eta)\big(\varphi(\eta)-\varphi(\xi)\big)\,d\eta=- \int_{Y^\sharp}  a(\xi-\eta)(\eta-\xi)\,d\eta, \quad \xi \in Y.
\end{equation}
Let us note that \eqref{efp} is a vector equation; equations for each component  $\varphi_j, \ j=1, \ldots d,$ of the function  $\varphi$ are uncoupled and they have the same form as \eqref{efp}:
\begin{equation*}\label{eq_for_phi}
  \int_Y \widetilde{a}(\xi-\eta)\big(\varphi_j(\eta)-\varphi_j(\xi)\big)\,d\eta=- \int_{Y^\sharp}  a(\xi-\eta)(\eta-\xi)_j\,d\eta, \quad \xi \in Y.
\end{equation*}
So we omit further index $j$.

From our assumptions on $a(\cdot)$ it follows that the function in the right-hand side of \eqref{efp} is periodic and bounded on $Y^\sharp$. Moreover, the symmetry of $a(\cdot)$ implies
$$
 \int_{Y}\int_{Y^\sharp} a(\xi-\eta)(\eta-\xi)\,d\eta d\xi=0.
$$
Denote the operator on the left-hand side of  \eqref{efp} by $A^Y$. Then $A^Y$ is a bounded linear
operator in $L^2(Y)$.
Notice that due to condition \eqref{a-1_low} and  regularity of $Y$ there exists $a_0>0$ such that
\begin{equation}\label{a0}
\alpha(\xi):= \int_Y \widetilde{a}(\xi-\eta)\,d\eta\geq a_0>0
\end{equation}
for all $\xi \in Y$.
Therefore, $A^Y$ is a Fredholm operator. Since
$$
(A^Y \varphi,\varphi)_{L^2(Y)}= \int_Y  \int_Y \widetilde{a}(\xi-\eta)\big(\varphi(\eta)-\varphi(\xi)\big)^2\,d\eta d\xi,
$$
the kernel of $A^Y$ consists of constant functions only. By the Fredholm theorem equation \eqref{efp}
has a unique (up to an additive constant) solution $\varphi\in L^2(Y)$.
Using \eqref{a0} and Corollary 5.5. in \cite{PSSZ} we conclude that $\varphi\in L^\infty(Y) \cap L^2(Y)$, i.e.  function $\varphi$ is bounded. Finally, using the same arguments as in Lemma \ref{tilde} we derive from \eqref{efp} that  $\varphi$ is a continuous function: $\varphi\in C(Y)$.

We turn to constructing the matrix valued function $\varkappa (\xi) = \{ \varkappa_{i,j} (\xi) \}, \ \xi \in Y, \ i,j=1,\ldots d$.
It is defined as a periodic solution of the equation
\begin{equation}\label{eq_for_theta}
A^Y \varkappa(\xi)=\Theta-\int\limits_{Y^\sharp}a(\xi-\eta)\big[\frac12(\xi-\eta)\otimes(\xi-\eta)-
(\xi-\eta)\otimes\varphi(\eta)\big]\,d\eta
\end{equation}
with $\varphi$ defined by \eqref{efp}. Again as above, \eqref{eq_for_theta} is the vector equation, and the equation for each component $\varkappa_{i,j} (\xi)$ of $\varkappa (\xi)$ is uncoupled and it has the same form as  \eqref{eq_for_theta}.
 The compatibility condition reads
 \begin{equation}\label{def_effmatr}
 |Y|\Theta_{i,j} = \int\limits_{Y^\sharp}\int\limits_{Y}a(\xi-\eta)\big[\frac12(\xi-\eta)_i (\xi-\eta)_j -
(\xi-\eta)_i \varphi_j(\eta)\big]\,d\eta d\xi.
\end{equation}
Under this choice of $\Theta$ equation \eqref{eq_for_theta} has a periodic solution
$\varkappa\in \big(L^2(Y)\big)^{d^2}$. In the same way as above one can show that
$\varkappa_{i,j}\in L^\infty(Y)$ and then using the same arguments as for the functions $\varphi$ we conclude that $\varkappa_{i,j}\in C(Y)$ for every $i,\ j$,

Moreover, \eqref{def_effmatr} implies that $\Theta$ is a positive definite matrix. This fact easily follows from the identity
$$
 |Y|\Theta=\frac12\int\limits_{Y^\sharp}\int\limits_{Y}a(\xi-\eta)\big[\varphi(\xi)-\varphi(\eta)-(\xi-\eta)\big]\otimes
\big[\varphi(\xi)-\varphi(\eta)-(\xi-\eta)\big] \,d\eta d\xi
$$
using the similar arguments as in \cite{JSP}.

Thus we proved the existence of bounded functions $\varphi_i$ and $\varkappa_{i,j}$ satisfying the equations \eqref{efp} and \eqref{eq_for_theta} respectively. These functions will play the role of correctors when constructing the function $F_\varepsilon$ in \eqref{4}.
By the definition \eqref{hat-L0eps} 
of the operator $\hat L_\varepsilon^0 $
we have
\begin{equation}\label{zzzzz}
\hat L_\varepsilon^0 F_\varepsilon^P (x)=\frac1{\eps^{d+2}}\int\limits_{\eps Y^\sharp}a(\frac{x-y}{\eps})\
\big(F_\varepsilon^P(y)-F_\varepsilon^P(x) \big)\,dy.
\end{equation}
To prove the relation in \eqref{PP1} we  substitute  the expression \eqref{FP} for $F_\varepsilon^P$ in \eqref{zzzzz} and expand $f_0(y) = f_0 (x + \eps z), \ z = \frac{y-x}\eps,$ in Taylor series
as follows
$$
f_0 (x + \eps z)=f_0(x)+\eps z\cdot\nabla f_0(x)+\frac {\eps^2}2
 z\otimes z\cdot\nabla\nabla f_0(x)+ \frac{\eps^3}{3!}z^{\otimes 3} \cdot \nabla \nabla\nabla f_0(x + \eps z \vartheta).
$$
Applying exactly the same arguments as in the proof of Proposition 3 in \cite{JSP} and taking into account \eqref{efp}, \eqref{eq_for_theta} and \eqref{def_effmatr} we obtain
 $$
\hat L_\varepsilon^0 F_\varepsilon^P (x)=\Theta \cdot\nabla\nabla f_0(x)+R_\eps(x),
$$
where $R_\eps$ is such that  $\|R_\eps\|_{L^\infty(Y^\sharp)} \to 0$ as $\eps\to0$.
This complete the proof of Proposition \ref{P1}.

\end{document}